\documentclass[a4paper, 12pt]{article}

\usepackage{amssymb}
\usepackage{bm}
\usepackage{bbm}
\usepackage{textcomp}
\usepackage{mathrsfs}
\usepackage{amsmath}
\usepackage{tcolorbox}
\usepackage{amsfonts}
\usepackage[T1]{fontenc}
\usepackage[latin9]{inputenc}
\usepackage{amsthm}
\usepackage{graphicx}
\usepackage{amsmath}
\newcommand{\Lfloor}{\left\lfloor}
\newcommand{\Rfloor}{\right\rfloor}

\usepackage{amsthm}
\usepackage{array}
\usepackage{cases}
\usepackage{enumerate}
\usepackage{clrscode}
\usepackage{listings}
\usepackage[ruled,vlined,english,linesnumbered]{algorithm2e}
\usepackage{url}
\usepackage{xcolor}  
\usepackage{algpseudocode}  

\usepackage{enumerate,enumitem,todonotes}
\makeatletter
\def\th@plain{%
  \itshape 
}
\makeatother

\makeatletter
\renewenvironment{proof}[1][\proofname]{\par
  \pushQED{\qed}%
  \normalfont \topsep6\p@\@plus6\p@\relax
  \trivlist
  \item[\hskip\labelsep
        \bfseries
    #1\@addpunct{.}]\ignorespaces
}{%
  \popQED\endtrivlist\@endpefalse
}
\makeatother

\newtheorem{theorem}{Theorem}[section]

\newtheorem{conjecture}[theorem]{Conjecture}
\numberwithin{equation}{section}

\newtheorem{thm}{Theorem}[section]
\newtheorem{cor}[thm]{Corollary}
\newtheorem{claim}{Claim}
\newtheorem{lem}[thm]{Lemma}

\numberwithin{equation}{section}

\usepackage[varg]{txfonts}
\usepackage{graphicx, epsfig, subfigure}
\usepackage{enumerate} 
\usepackage[square, numbers, sort&compress]{natbib} 
\ifx\pdfoutput\undefined
 \usepackage[dvipdfm,%
  pdfstartview=FitH, 
  bookmarks=true,%
  bookmarksnumbered=true, 
  bookmarksopen=true, 
  plainpages=false,%
  pdfpagelabels,%
  colorlinks=true, 
  linkcolor=blue, 
  citecolor=blue,%
  urlcolor=black,
  pdfborder=001]{hyperref}
  \AtBeginDvi{}  
\else
 \usepackage[pdftex,%
  pdfstartview=FitH, 
  bookmarks=true,%
  bookmarksnumbered=true, 
  bookmarksopen=true, 
  plainpages=false,%
  pdfpagelabels,%
  colorlinks=true, 
  linkcolor=blue, 
  citecolor=blue,%
  urlcolor=black,
  pdfborder=001]{hyperref}
\fi

\numberwithin{equation}{section}

\begin{document}
\title{\LARGE Equitable coloring of sparse graphs}
\author{Weichan Liu$^b$\thanks{Supported by the Postdoctoral Fellowship Program of CPSF under Grant Number GZC20252020.} \quad\quad Xin Zhang$^a$\footnote{Corresponding author. Emails: xzhang@xidian.edu.cn (XZ), wcliu@sdu.edu.cn (WL).}\\
{\small a. School of Mathematics and Statistics, Xidian University, Xi'an, 710071, China}\\
{\small b. School of Mathematics, Shandong University, Jinan, 250100, China}\\ 
}


\maketitle

\begin{abstract}\baselineskip 0.60cm
An equitable coloring of a graph is a proper coloring where the sizes of any two distinct color classes differ by at most one.
The celebrated Chen-Lih-Wu Conjecture (CLWC for short) states that every connected graph $G$ that is neither an odd cycle, a $K_r$, nor a $K_{2m+1,2m+1}$ has an equitable $\Delta(G)$-coloring.
A graph $G$ is in $\mathcal{G}_{m_1,m_2}$ if for all $H\subseteq G$, $\lVert H \rVert\leq m_1|H|$, and if $H$ is bipartite, then $\lVert H \rVert\leq m_2|H|$.
In this paper, we confirm CLWC for all graphs $G$ in $\mathcal{G}_{m_1, m_2}$ provided that $m_1\leq 1.8m_2$ and $\Delta(G)\geq \frac{2m_1}{1-\beta}$, where $\beta$ is a real root of $2m_2(1-x)(1+x)^2-m_1x(2+x)$. By specializing to the case $m_1 = m_2 = d$, we deduce that every $d$-degenerate graph $G$ with $\Delta(G) \geq 6.21d$ admits an equitable $r$-coloring for all $r \geq \Delta(G)$, thereby improving the previous best-known lower bound of $10d$ on $\Delta(G)$ established by Kostochka and Nakprasit in 2005.
A graph is $k$-planar if it can be drawn in the plane so that each edge is crossed at most $k$ times. CLWC had been confirmed for planar graphs $G$ with $\Delta(G) \geq 8$ (Kostochka, Lin, and Xiang, 2024) and for $1$-planar graphs $G$ with $\Delta(G) \geq 13$ (Cranston and Mahmoud, 2025). As an immediate application of our main result, we extend this confirmation to all $k$-planar graphs $G$ with $k \geq 2$ and $\Delta(G) \geq \sqrt{383k}$.

\vspace{1em}

\end{abstract}

\baselineskip 0.60cm

\maketitle

\section{Introduction}

An \textit{equitable $r$-coloring} of a graph $G$ is a partition of $V(G)$ into $r$ independent sets $V_1,\ldots,V_r$ so that $\big||V_i|-|V_j|\big|\leq 1$ for each $1\leq i<j\leq r$.
We can image that all vertices of $V_i$ are colored by color $i$, thereby rendering $V_i$ a \textit{color class} of this equitable $r$-coloring.

The notion of equitable coloring of graphs was first proposed by Meyer \cite{Meyer1973} in 1973.
The motivation for its proposal comes from Tucker's paper \cite{doi:10.1137/1015072}, which studies the graph theory model of the problem of arranging garbage collection routes reasonably: using vertices to represent garbage collection routes, two vertices are adjacent if and only if the garbage collection routes they represent are not running on the same day. Meyer expects to run roughly equal numbers of garbage collection routes on each of the six workdays of the week.

By employing a greedy algorithm, it can be readily demonstrated that any graph $G$ has a proper coloring using at most $\Delta(G)+1$ colors.
Motivated by this, Erd\H{o}s  \cite{erdos1959} conjectured any graph $G$ has an equitable $r$-coloring for every integer $r\geq \Delta(G) + 1$.
This conjecture was confirmed by Hajnal and Szemer\'edi \cite{HS1970}. In 2008, Kierstead and Kostochka \cite{zbMATH05302004} simplified the proof of 
Hajnal-Szemer\'edi theorem, and later in 2010, Kierstead, Kostochka, Mydlarz, and Szemer{\'e}di  \cite{zbMATH05881225} presented an algorithmic proof, which implies that such an equitable coloring can be found in $O(|G|^2 r)$ time.

In most instances, $\Delta(G)+1$ colors are unnecessarily numerous for achieving a proper coloring of a graph $G$.
The celebrated  Brooks' theorem states that if $G$ is neither an odd cycle nor a complete graph, then $\Delta(G)$ colors are sufficient to color the vertices of $G$ properly. Perhaps drawing inspiration from this, Meyer \cite{Meyer1973} posed the question of whether the equitable variant of Brooks' theorem holds true. Nevertheless, this is not so, as evidenced by the graph $K_{2m+1,2m+1}$, which has a maximum degree of $2m+1$ yet does not allow for an equitable $(2m+1)$-coloring.
In 1994, Chen, Lih, and Wu \cite{CHEN1994443} proposed the following conjecture, which is the main open question in the area of equitable coloring. 
\begin{conjecture}[Chen-Lih-Wu Conjecture]
    If $G$ is a connected graph that is not an odd cycle, a complete graph, or a complete bipartite graph of the form 
    $K_{2m+1,2m+1}$, then it has an equitable $r$-coloring for every $r\geq \Delta(G)$.
\end{conjecture}

Until now, Chen-Lih-Wu Conjecture had been confirmed for graphs $G$ with $\Delta(G)\geq |G|/2$ \cite{CHEN1994443}, with $(|G|+1)/3\leq \Delta(G)< |G|/2$ \cite{chen2014equitable},
with $\Delta(G)\geq |G|/4$ \cite{zbMATH06699287}, with $\Delta(G)=3$ \cite{CHEN1994443}, or with  $\Delta(G)=4$ \cite{zbMATH06081427}.

For planar graphs $G$, 
Zhang and Yap \cite{zbMATH01308943} demonstrated that if $\Delta(G) \geq 13$, then $G$ admits an equitable $\Delta(G)$-coloring. 
Nakprasit \cite{NAKPRASIT20121019} extended this result to the range $9 \leq \Delta(G) \leq 12$. 
More recently, Kostochka, Lin, and Xiang \cite{zbMATH07840571} further improved the bound to $\Delta(G) \geq 8$.

For non-planar graphs, the investigation into ``\textit{Graph Drawing Beyond Planarity}" has emerged as a rapidly advancing research field, which concentrates on classifying and dissecting their geometric representations, and with a particular focus on identifying and analyzing forbidden crossing patterns. 
Currently, 1-planar graphs have gained significant attention as a notable class within the broader category of graphs beyond planarity, as outlined in a survey paper by Kobourov, Liotta, and Montecchiani \cite{zbMATH06782829}.
In this context, a graph is \textit{1-planar} if it has a drawing in the plane where each edge intersects no more than one other edge.

The equitable coloring of 1-planar graphs was initially investigated by Zhang \cite{zbMATH06602931}. 
By considering the class $\mathcal{G}$ of graphs $G$ where every subgraph $H$ of $G$ satisfies $\lVert H \rVert\leq 4|H|-8$, which is a broader class then 1-planar graphs,
he showed that every graph $G\in \mathcal{G}$ with $\Delta(G)\geq 17$ admits an equitable $\Delta(G)$-coloring. 
Subsequently, Zhang, Wang, and Xu \cite{zbMATH06875705} refined the result for 1-planar graphs, showing that every such graph with
$\Delta(G) \geq 15$ admits an equitable $\Delta(G)$-coloring. Most recently, Cranston and Mahmoud \cite{CRANSTON2025114286} further tightened the bound to $\Delta(G) \geq 13$ for 1-planar graphs and suggests adapting their proof to work for all graphs in $\mathcal{G}$. 
The techniques utilized in \cite{CRANSTON2025114286} draw parallels to those employed in \cite{zbMATH07840571}, which hinge on the edge density for bipartite 1-planar graphs or planar graphs and the 1-planarity of $K_6$ or the planarity of $K_4$. Motivated by this line, we embark on an investigation into the equitable coloring of even more general graph classes.

Let $\mathcal{G}_{m_1,m_2}$ be the class of graphs $G$ satisfying:
\begin{itemize}
    \item $\lVert H \rVert\leq m_1 |H|$ for every subgraph $H$ of $G$,
    \item $\lVert H \rVert\leq m_2 |H|$ for every bipartite subgraph $H$ of $G$,
\end{itemize}
where $m_1$ and $m_2$ are real numbers, and $\lVert H \rVert$ and $|H|$ are the number of edges and vertices in $H$, respectively. 
It is evident that $\mathcal{G}_{m_1,m_2}$ is closed under taking subgraph.
When $m_1=m_2=d$, we denote $\mathcal{G}_{m_1,m_2}$ as $\mathcal{G}_d$.
The main contribution of this paper lies in establishing the following theorem.

\begin{thm} \label{thm-main}
Let $G\in \mathcal{G}_{m_1,m_2}$ with $m_1\leq 1.8m_2$. If 
\[\Delta(G)\geq \frac{2}{1-\beta}m_1,\]
where $\beta$ is a real root of 
\[ 
2m_2(1-x)(1+x)^2-m_1x(2+x),
\]
then $G$ is equitably $r$-colorable for every $r\geq \Delta(G)$.
\end{thm}

\noindent Taking $m_1=m_2=d$ in Theorem \ref{thm-main}, we arrive at the following theorem as an immediate corollary.

\begin{thm} \label{thm-maincorollary}
Let $G\in \mathcal{G}_{d}$. If $\Delta(G)\geq 6.21d$,
then $G$ is equitably $r$-colorable for every $r\geq \Delta(G)$.
\end{thm}

It is noteworthy that Theorem \ref{thm-maincorollary} holds independent interest in its own right.
A \textit{$d$-degenerate graph} is a graph in which every subgraph has a vertex of degree at most $d$.
Kostochka and Nakprasit \cite{zbMATH01877116} proved that
every $d$-degenerate graph ($d\geq 2$) with maximum degree at most $\Delta$
is equitably $k$-colorable, where $k =  (d + \Delta + 1)/2$, under the condition that
$\Delta \geq 27d$. This implies that every $d$-degenerate graph with maximum degree at most $\Delta$ is equitably $\Delta$-colorable provided $\Delta\geq 14d+1$. 
This bound was improved to $10d$ by Kostochka and Nakprasit \cite{zbMATH05013296}, who actually showed that 
for every graph on at least 46 vertices and with maximum degree at most $\Delta$, if $e(G)\leq \Delta v(G)/5$ and $K_{\Delta+1}$ is not a subgraph of $G$, then $G$ has an equitable $\Delta$-coloring.

 Observe that for every $d$-degenerate graph, the number of edges is less than $d$ times the number of vertices, and therefore, the class of $d$-degenerate graphs falls within the class $\mathcal{G}_d$. Consequently, Theorem \ref{thm-maincorollary} implies the following.

 \begin{thm}
Every $d$-degenerate graph $G$ with $\Delta(G)\leq \Delta$ can be equitably $\Delta$-colorable if $\Delta \geq 6.21d$. 
 \end{thm}
 
\noindent This finding refines and extends the result obtained by Kostochka and Nakprasit \cite{zbMATH01877116,zbMATH05013296}.


\section{Preliminaries}


From this point onward, we uniformly assume $m_1 \leq 1.8m_2$. For a graph $H$, given a subset $S \subseteq V(H)$ and a vertex $v \in V(H) \setminus S$, the notation $e_H(v, S)$ represents the number of neighbors of $v$ in $H$ that belong to $S$. Furthermore, for two disjoint subsets $A, B \subseteq V(H)$, $e(A, B)$ denotes the number of edges that exist between $A$ and $B$. 

\begin{lem} \label{lem:denenerate}
    Every graph $G\in \mathcal{G}_{m_1,m_2}$ is $\lfloor 2m_1 \rfloor$-degenerate.
\end{lem}

\begin{proof}
Since $\lVert H \rVert\leq m_1 |H|$ for every subgraph $H$ of $G$, the average degree of $H$ is at most $2m_1$, and thus $\delta(H)\leq \lfloor 2m_1 \rfloor$.
\end{proof}

\begin{lem} \label{lem:large-complete-graph}
 $K_t\in \mathcal{G}_{m_1,m_2}$ for every $1\leq t\leq \lfloor 2m_1 \rfloor$.
\end{lem}

\begin{proof}
If $H$ is a subgraph of $K_t$,
then 
\begin{align*}
    \lVert H \rVert\leq \frac{|H|(|H|-1)}{2}\leq m_1|H|
\end{align*} 
since $|H|\leq t\leq 2m_1$.
If $H$ is a bipartite subgraph of $K_t$ with two parts $A$ and $B$, then
\begin{align*}
    \lVert H \rVert\leq |A|\cdot |B|\leq \frac{(|A|+|B|)^2}{4}=\frac{|H|^2}{4}\leq m_2|H|
\end{align*}
since $|H|\leq t\leq 2m_1< 4m_2$.
\end{proof}

\begin{lem} \label{lem:graph-union-in-the-class}
    If $G\in \mathcal{G}_{m_1,m_2}$, then $G\cup K_t \in \mathcal{G}_{m_1,m_2}$ for every $1\leq t\leq \lfloor 2m_1 \rfloor$.
\end{lem}

\begin{proof}
   Any subgraph $H$ of $G\cup K_{\lfloor 2m_1 \rfloor}$ can be written as $H_1\cup H_2$ such that 
   $H_1\subseteq G$ and $H_2\subseteq K_{\lfloor 2m_1 \rfloor}$. It is possible that $H_1$ or $H_2$ is a null graph.
   Now by the definition of $\mathcal{G}_{m_1,m_2}$ and by Lemma \ref{lem:large-complete-graph},
   \begin{align*}
       \lVert H \rVert=\lVert H_1 \rVert+\lVert H_2 \rVert\leq m_1 |H_1|+m_1|H_2|=m_1 |H|.
   \end{align*}
   Similarly, if $H$ is a bipartite subgraph of $G\cup K_{\lfloor 2m_1 \rfloor}$, then $H=H_3\cup H_4$, where 
   $H_3$ is a bipartite subgraph of $G$ and $H_4$ is a bipartite subgraph of $K_{\lfloor 2m_1 \rfloor}$. Hence
   \begin{align*}
       \lVert H \rVert=\lVert H_3 \rVert+\lVert H_4 \rVert\leq m_2 |H_3|+m_2|H_4|=m_2 |H|
   \end{align*}
    by the definition of $\mathcal{G}_{m_1,m_2}$ and by Lemma \ref{lem:large-complete-graph}.
\end{proof}

\begin{lem} \label{lem:divisiblecase}
 If every graph $H\in \mathcal{G}_{m_1,m_2}$ with $r\mid |H|$ is equitably $r$-colorable,
 then every graph $G\in \mathcal{G}_{m_1,m_2}$ is equitably $r$-colorable.
\end{lem}

\begin{proof}
Let $n:=|G|= rs-t$ where $0\leq t\leq r-1$. If $t=0$, then $r \mid v(G)$ and the result naturally holds. If $1\leq t\leq {\lfloor 2m_1 \rfloor}$, then set $H:=G\cup K_t$. Since $H\in \mathcal{G}_{m_1,m_2}$ by Lemma \ref{lem:graph-union-in-the-class} and
$|H|=rs$, $H$ has an equitable $r$-coloring $\varphi$ with the size of each color class being $s$. Restricting $\varphi$ to $G$ we obtain an equitable $r$-coloring of $G$. In the following, we assume ${\lfloor 2m_1 \rfloor}<t\leq r-1$.

By Lemma \ref{lem:denenerate}, the vertices of $G$ can be arranged in a sequence $v_1,v_2,\ldots,v_{n}$ so that $v_i$ has at most ${\lfloor 2m_1 \rfloor}$ neighbors among $\{v_1,\ldots,v_{i-1}\}$ for each $2\leq i\leq n$. Let $H$ be the subgraph induced by $\{v_1,v_2,\ldots,v_{r(s-1)}\}$ and let $W=\{v_{r(s-1)+1},\ldots,v_{r(s-1)+(r-t)}\}$.
Since $H\in \mathcal{G}_{m_1,m_2}$ and $|H|=r(s-1)$, $H$ has an equitable $r$-coloring $\phi$ with the size of each color class being $s-1$. We extend $\phi$ to an equitable $r$-coloring of $G$ by properly coloring
vertices in $W$ in the ordering $v_{r(s-1)+1},\ldots,v_{r(s-1)+(r-t)}$ so that they receive distinct colors.
Assume now that $v_{r(s-1)+i}$ is being colored for some $1\leq i\leq r-t$. Since $v_{r(s-1)+i}$ has at most  
${\lfloor 2m_1 \rfloor}$ colored neighbors in $H$ and we had already colored $i-1$ vertices of $S$, 
there are at least $r-{\lfloor 2m_1 \rfloor}-(i-1)>r-t-(r-t-1)=1$ available colors for $v_{r(s-1)+i}$. Therefore, vertices in $W$ can be colored as desired.
\end{proof}

Let \( S = \left\{ x\in [0.5,1)\;\middle|\; G(x) := \dfrac{2(1-x)(1+x)^2}{x(x+2)} \geq \dfrac{m_1}{m_2} \right\} \) and define 
\begin{align}\label{def-beta}
    \beta=\max\left\{x~|~x\in S\right\}.
\end{align}
Given that $G(x)$ is continuous and monotonically decreasing on the interval $x\in [0.5,1)$, with $G(0.5) = 1.8$, $G(1) = 0$, and $\frac{m_1}{m_2} < 1.8$, it follows that
$$
G(\beta) = \frac{m_1}{m_2},
$$
i.e.\, $\beta$ is a real root of $2m_2(1-x)(1+x)^2 - m_1x(2+x)$.
Let 
\begin{align} \label{def-r0}
    r_0= \frac{2}{1-\beta}m_1.
\end{align} and 
\begin{align}\label{def-a0}
    a_0= \frac{r_0-\sqrt{r_0^2-4m_2r_0}}{2}.
\end{align}

\begin{lem}
Let $m_1,m_2$ be positive numbers such that $1 \leq m_1/m_2 \leq 1.8$ and let $\beta,r_0,a_0$ be numbers defined by \eqref{def-beta}, \eqref{def-r0}, and \eqref{def-a0}.
Now, we have
\begin{align}
& \quad \frac{r_0 + \sqrt{r_0^2 - 4m_2r_0}}{2} \geq 2m_1, \label{condition-I} \\
&  \quad r_0 \geq 3a_0 + 1, \label{condition-II}\\
&  \quad (1 - \beta^2)(r_0 - 2a_0) \geq 2m_1, \label{condition-III}\\
&   \quad \beta(r_0 - 2a_0) > 2a_0. \label{condition-IV}
\end{align}
\end{lem}

\begin{proof}

If \eqref{condition-IV} is met, then $r_0>\left(2+\frac{2}{\beta}\right)a_0> 3a_0$, implying \eqref{condition-II}.
By \eqref{def-a0}, \eqref{condition-IV} is equivalent to 
$\beta r_0\geq 2a_0(\beta+1)=(\beta+1)\left(r_0-\sqrt{r_0^2-4m_2r_0}\right)$,
which is further equivalent to 
\begin{align} \label{bound:r0-second}
    r_0\geq \frac{4(1+\beta)^2}{\beta^2+2\beta} m_2,
\end{align}
and \eqref{condition-III} is equivalent to the inequality
\[
(1-\beta^2)\sqrt{r_0^2 - 4m_2r_0} \geq 2m_1,
\]
which  holds provided that
\begin{equation} \label{bound:r0-third}
    r_0 \geq 2m_2 + 2\sqrt{m_2^2 + \frac{m_1^2}{(1-\beta^2)^2}}.
\end{equation}

Since
\begin{align*}
     &\frac{4(1+\beta)^2}{\beta^2+2\beta} m_2 \geq  2m_2 + 2\sqrt{m_2^2 + \frac{m_1^2}{(1-\beta^2)^2}} 
    \Leftrightarrow  \bigg(\frac{\beta^2+2\beta+2}{\beta^2+2\beta} \bigg)^2m_2^2\geq m_2^2 + \frac{m_1^2}{(1-\beta^2)^2}\\
    \Leftrightarrow  & \bigg(\frac{2(\beta+1)}{\beta^2+2\beta} \bigg)^2 m_2^2 \geq \frac{m_1^2}{(1-\beta^2)^2}
    \Leftrightarrow \dfrac{2(1-\beta)(1+\beta)^2}{\beta(\beta+2)} \geq \dfrac{m_1}{m_2}
    \overset{\beta\geq 0.5}{\Rightarrow} \frac{4(1+\beta)^2}{\beta^2+2\beta} m_2 \geq 4m_1,
\end{align*}
\eqref{bound:r0-second} entails \eqref{bound:r0-third} and $r_0\geq 4m_1$ (implying \eqref{condition-I}) by the definition of $\beta$. Thus,
\begin{align*}
     r_0=\frac{4(1+\beta)^2}{\beta^2+2\beta} m_2=\frac{2}{1-\beta}G(\beta)m_2=\frac{2}{1-\beta}m_1
\end{align*}
satisfies all the inequalities.
\end{proof}

Rather than proving Theorem \ref{thm-main} directly, we show a slightly stronger result as follows. Although it is equivalent to Theorem \ref{thm-main} in light of the Hajnal-Szemer\'edi Theorem, the condition "\(\Delta(G) \leq r\)" in this theorem exhibits subgraph hereditarity, which is crucial for our proof by induction.
\begin{thm} \label{thm-mainfinal}
Let $G\in \mathcal{G}_{m_1,m_2}$ with $m_1\leq 1.8m_2$ and let $r$ be an integer such that $r\geq r_0$. If $\Delta(G)\leq r$, then $G$ is equitably $r$-colorable.
\end{thm}

\section{Proof of Theorem \ref{thm-mainfinal}}

According to Lemma \ref{lem:divisiblecase}, it suffices to prove the result in the case when $|G|=rs$. 
We utilize induction on $\lVert G \rVert$ to proceed with the proof. Clearly, the result trivially holds if $\lVert G \rVert=0$. Now we assume $\lVert G \rVert\geq 1$.
Choose an edge $xy$ so that $\deg_G(x)\leq \lfloor 2m_1 \rfloor$. This edge exists because of Lemma \ref{lem:denenerate}.
Let $G':=G-xy$. Since $G',G\in \mathcal{G}_{m_1,m_2}$, $|G'|=rs$ and $\Delta(G')\leq \Delta(G)\leq r$, $G'$ has an equitable $r$-coloring 
$\varphi$ by the induction hypothesis. If $\varphi(x)\neq \varphi(y)$, then $\varphi$ also gives an equitable $r$-coloring of $G$, and we are done. Otherwise, $\varphi(x)= \varphi(y)$. Now let $H:=G'-x$ and restrict $\varphi$ to $H$.
This gives an equitable $r$-coloring of $H$, still denoted by $\varphi$, so that all color classes but the one containing $y$ (which has size $s-1$, called a \textit{small class}) have size $s$. Let $V_1$ be the color class such that $y\in V_1$ and let $V_2,\ldots,V_r$ be other color classes.

Construct a digraph $\mathcal{D}:=\mathcal{D}(\varphi)$ on the vertex set $\{V_1,V_2,\ldots,V_r\}$ so that $\overrightarrow{V_iV_j}\in E(\mathcal{D})$ if and only if there is a vertex $v\in V_i$ such that $e_H(v,V_j)=0$. We call such a vertex $v$ \textit{movable} to $V_j$ and say that $v$ \textit{witnesses} $\overrightarrow{V_iV_j}$ .
If $P:=\overrightarrow{V_{j_1}V_{j_2}\cdots V_{j_k}}$ is a directed path in $\mathcal{D}$ and $v_i$ is a vertex in $V_{j_i}$ such that
$v_i$ witnesses  $\overrightarrow{V_{j_i}V_{j_{i+1}}}$, then \textit{switching witnesses along $P$} means moving 
$v_i$ to $V_{j_{i+1}}$ for every $1\leq i<k$. This operation decreases $V_{j_1}$ by one and
increases $V_{j_k}$ by one, while leaving the sizes of the interior vertices (color classes) unchanged.
For some $V_j$, if $j=1$ or there is a directed path from $V_j$ to $V_1$ in $\mathcal{D}$, then we say that 
$V_j$ is \textit{accessible}.
If $V_j$ ($j\neq 1$) is an accessible class and for each accessible class $V_k$ with $k\neq j$ there is a directed path from $V_k$ to $V_1$
in $\mathcal{D}-V_j$, then we call such $V_j$ a \textit{terminal}. We stipulate that $V_1$ is also a terminal if $V_1$ is a unique accessible class.

Let $\mathcal{A}$ be the set of all accessible classes and let $\mathcal{B}$ be the set of other classes.
Set
\begin{align*}
    a:=a(\varphi)=|\mathcal{A}|, ~ b:=b(\varphi)=|\mathcal{B}|.
\end{align*}
Note that $a$ is completely determined by $\varphi$ and that $a\geq 1$. We may choose $\varphi$ so that $a$ is maximum.

Let $A=\bigcup \mathcal{A}$ and $ B=\bigcup \mathcal{B}$. Now,
\begin{align*}
    |A|=as-1, ~ |B|=bs.
\end{align*}

If $a\geq \lfloor 2m_1 \rfloor+1$, then there is one class $V_j\in \mathcal{A}$ such that $e_{G}(x,V_j)=0$ since $\deg_G(x)\leq \lfloor 2m_1 \rfloor$. Since $y\in V_1$ and $xy\in E(G)$, we have $j\neq 1$. Let $P$ be the directed path from $V_j$ to $V_1$ in $\mathcal{D}$; we call such a path a \textit{$(V_j,V_1)$-path} in the following. 
Now, put $x$ into $V_j$ and switch witnesses along $P$. This gives an equitable $r$-coloring of $G$. 

Otherwise,
\begin{align} \label{eq:rough-bound-a}
    a\leq \lfloor 2m_1 \rfloor.
\end{align}
For each pair of classes $V_i\in \mathcal{A}$ and $V_j\in \mathcal{B}$, $e_H(v,V_i)\geq 1$ for every $v\in V_j$ by the definitions of $\mathcal{A}$ and $\mathcal{B}$. This implies 
\begin{align*}
    a(r-a)s=abs\leq e_H(A,B)\leq m_2(|A|+|B|)=m_2\left((a+b)s-1\right)=m_2(rs-1)<m_2 rs,
\end{align*}
where the second inequality arises due to the fact that the edges connecting $A$ and $B$ form a bipartite graph,
and thus 
\begin{align} \label{eq:function-a}
    a^2-ra+m_2r>0.
\end{align}

If $a>m_2$, then \eqref{eq:function-a} implies $a^2-r_0a+m_2r_0>0$, and thus
\begin{align}\label{eq:good-bound-a}
    a\leq \lfloor a_0 \rfloor
\end{align}
by \eqref{eq:rough-bound-a} and \eqref{condition-I}.
Else $a\leq m_2$ and \eqref{eq:good-bound-a} also holds since $m_2<a_0$ by \eqref{def-a0}. 

A set $\mathcal{C}\subseteq V(\mathcal{D})$ is a \textit{strong component}  if 
for every two distinct classes 
$V_i,V_j\in \mathcal{C}$ there is a $(V_i,V_j)$-path in $\mathcal{D}$. 
We show that the digraph induced by $\mathcal{B}$ contains a large strong component.

\begin{claim} \label{claim:strong-component}
There is a subset $\mathcal{C}\subseteq \mathcal{B}$ with $|\mathcal{C}|\geq r-\lfloor a_0 \rfloor$ forming a strong component in $\mathcal{D}[\mathcal{B}]$.
\end{claim}

\begin{proof}
Suppose, to the contrary, that every strong component in $\mathcal{D}[\mathcal{B}]$ has at most
$r-\lfloor a_0 \rfloor-1$ classes. If there is a strong component $\mathcal{C}_0$ in $\mathcal{D}[\mathcal{B}]$ such that $|\mathcal{C}_0|\geq \lfloor a_0 \rfloor+1$, then let $\mathcal{Z}:=\mathcal{C}_0$. If  every strong component in $\mathcal{D}[\mathcal{B}]$ has at most $\lfloor a_0 \rfloor$ classes, then assume that $\mathcal{D}[\mathcal{B}]$ has exactly $N$ strong strong components $\mathcal{C}_1,\ldots,\mathcal{C}_N$ such that $|\mathcal{C}_1|\leq \cdots \leq |\mathcal{C}_N|$.
Let $M$ be the first integer such that $|\mathcal{C}_1|+\cdots +|\mathcal{C}_M|\geq \lfloor a_0 \rfloor+1$.
If $|\mathcal{C}_1|+\cdots +|\mathcal{C}_M|\geq r-\lfloor a_0 \rfloor$, then 
$|\mathcal{C}_1|+\cdots +|\mathcal{C}_{M-1}|\geq r-\lfloor a_0 \rfloor-|\mathcal{C}_M|\geq r-2\lfloor a_0 \rfloor\geq r_0-2\lfloor a_0 \rfloor\geq a_0+1$ by \eqref{condition-II}, contradicting the choice of $M$. Therefore, in each case we find a union $\mathcal{Z}$ of some strong components  in $\mathcal{D}[\mathcal{B}]$ such that $\lfloor a_0 \rfloor+1\leq z:=|\mathcal{Z}|\leq r-\lfloor a_0 \rfloor-1$. 

Let
$Z=\bigcup \mathcal{Z}$.
For every class $V_i\in \mathcal{B}\setminus \mathcal{Z}$ and every class $V_j\in \mathcal{Z}$, either 
$\overrightarrow{V_iV_j}\not\in E(\mathcal{D})$ or $\overrightarrow{V_jV_i}\not\in E(\mathcal{D})$.
This implies either $e_H(v,V_j)\geq 1$ for every $v\in V_i$ or $e_H(v,V_i)\geq 1$ for every $v\in V_j$.
In both cases we have $e_H(V_i,V_j)\geq s$. Therefore, 
\begin{align}\label{eq:cal-1}
    e_H(Z,B\setminus Z)\geq z(b-z)s.
\end{align}
Since $\mathcal{Z}\subset \mathcal{B}$, for every vertex $v$ in every class $V_i\in \mathcal{Z}$ we have 
$e_H(v,V_j)\geq 1$ for every $V_j\in \mathcal{A}$. It follows
\begin{align}\label{eq:cal-2}
    e_H(Z,A)\geq azs.
\end{align}
Since the graph induced by the edges between $Z$ and $A\cup B\setminus Z$ is bipartite, combining \eqref{eq:cal-1} and \eqref{eq:cal-2}, we conclude
\begin{align} \label{eq:for-z}
      (r-z)zs=(a+b-z)zs=z(b-z)s+azs\leq  e_H(Z,A\cup B\setminus Z)\leq m_2(rs-1)<m_2rs.
\end{align}
Since $\lfloor a_0 \rfloor+1\leq z\leq r-\lfloor a_0 \rfloor-1$, \eqref{eq:for-z} implies
\begin{align}\label{eq:a0+1}
    (\lfloor a_0 \rfloor+1)^2-r(\lfloor a_0 \rfloor+1)+m_2r>0.
\end{align}
Since $a_0>m_2$ by \eqref{def-a0},
\eqref{eq:a0+1} implies 
\begin{align} \label{eq:a0+1++}
    (\lfloor a_0 \rfloor+1)^2-r_0(\lfloor a_0 \rfloor+1)+m_2r_0>0.
\end{align}
By \eqref{def-a0} and \eqref{condition-II}, we have $r_0^2-4m_2r_0=(r_0-2a_0)^2\geq (a_0+1)^2\geq 1$, and thus 
 $\lfloor a_0 \rfloor+1\leq  a_0+1 \leq \frac{r_0+\sqrt{r_0^2-4m_2r_0}}{2}$. Now, \eqref{eq:a0+1++}
implies
\begin{align*}
    \lfloor a_0 \rfloor+1  \leq \Lfloor \frac{r_0-\sqrt{r_0^2-4m_2r_0}}{2} \Rfloor=\lfloor a_0 \rfloor,
\end{align*}
a contradiction. 
\end{proof}

Let $uv$ be an edge such that $u\in V_i\in \mathcal{A}$ and $v\in V_j\in \mathcal{B}$. 
If $e_H(v,V_i)=1$ (i.e.\,$u$ is the unique neighbor of $v$ in $V_i$), then we say that $uv$ is a \textit{solo edge}, $u$ is a \textit{solo vertex}, and $v$ is a \textit{solo neighbor} of $u$. 
If there is a pair of solo edges $uv$ and $uw$, where $u\in A$ and $v,w\in B$, such that $vw\not\in E(H)$, then we say that $v$ and $w$ are \textit{nice solo neighbors} of $u$. 
Note that if $uv$ is a solo edge, where $u\in V_i\in \mathcal{A}$, then $V_i+v-u$ is still an independent set.
We will frequently apply this fact in the remaining arguments.

\begin{claim}\label{claim:solo-at-least-2-neighbors-}
Let $u\in V_i\in \mathcal{A}$, where $V_i$ is a terminal. If $i=1$ or there is a vertex $u'\in V_i\setminus \{u\}$ such that $u'$ is movable to some other class $V_{i'}\in \mathcal{A}$, then for each $V_j\in \mathcal{B}$ containing a nice solo neighbor of $u$, we have
$e_H(u,V_j)\geq 2$.
\end{claim}

\begin{proof}
Let $v\in V_j$ be a nice solo neighbor of $u$. Because of $v$, $u$ has another nice solo neighbor $w$. If $w\in V_j$, then 
$e_H(u,V_j)\geq 2$ and we are done. Otherwise $w\in V_{j'}\in \mathcal{B}$ for some $j'\neq j$ and $e_H(u,V_j)=1$.
Now, we move $u$ to $V_j$, and move $v$ to $V_i$. This gives a new equitable $r$-coloring $\phi$ of $H$ as $V_i+v-u$ and $V_j+u-v$ are both independent sets. The small class of $\phi$ is still $V_1$ if $i\geq 2$, or $V_i+v-u$ if $i=1$.

Since $uw$ is a solo edge and $vw\not\in E(H)$, $e_H(w,V_i+v-u)=0$.
This implies $\overrightarrow{V_{j'}(V_i+v-u)}$ is an arc in $\mathcal{D}(\phi)$.
If $i = 1$, then $a(\phi) \geq 2 > 1 = a(\varphi)$, which contradicts the choice of $\varphi$. Otherwise, $i \neq 1$.
Since $V_i$ is a terminal under $\varphi$, each $V_{k}\in \mathcal{A}$ with $k\neq i$ is still an accessible class under $\phi$. Due to the existence of $u'$, $\overrightarrow{(V_i+v-u)V_{i'}}$ is an arc in $\mathcal{D}(\phi)$.
Therefore, all classes in $\mathcal{A}\setminus \{V_i\}$ and $V_i+v-u$ and $V_{j'}$ are accessible.
This implies $a(\phi)\geq a(\varphi)+1$, contradicting the choice of $\varphi$.
\end{proof}

Given a vertex $u\in A$, let $Q(u)$ be the set of solo neighbors of $u$ and let $Q'(u)$ be the set of nice solo neighbors of $u$. Set $q(u)=|Q(u)|$ and $q'(u)=|Q'(u)|$.
We show that if $q(u)$ is large then the proportion of nice solo neighbors of $u$, relative to all its solo neighbors, is at least $\beta$, where $\beta$ is defined by \eqref{def-beta}.

\begin{claim}\label{claim:counting-nice-solo}
Let $u$ be a vertex in $A$. If $q(u)\geq r_0-2\lfloor a_0 \rfloor$, then $q'(u)\geq \beta q(u)$.
\end{claim}

\begin{proof}
By the definition of nice solo neighbor, 
$vw\in E(H)$ for every pair of vertices $v$ and $w$ such that $v\in Q(u)\setminus Q'(u)$ and $w\in Q(u)$.
Counting edges in the subgraph of $H$ induced by $u\cup Q(u)$, we have
\begin{align} \label{eq:edges-u+Q(u)}
    m_1(q(u)+1)\geq \bigg\lVert H[u\cup Q(u)] \bigg\rVert =\binom{q(u)+1}{2}-\bigg\lVert \overline{H[Q'(u)]}\bigg\rVert\geq \binom{q(u)+1}{2}-\binom{q'(u)}{2}. 
\end{align}
Let 
\begin{align*}
    F(x):=(1-\beta^2)x^2-(2m_1-\beta-1)x-2m_1.
\end{align*}
We have
\begin{align}
    F(r_0-2\lfloor a_0 \rfloor) &=(1-\beta^2)(r_0-2\lfloor a_0 \rfloor)^2-(2m_1-\beta-1)(r_0-2\lfloor a_0 \rfloor)-2m_1, \notag \\
                                &=(r_0-2\lfloor a_0 \rfloor)\bigg( (1-\beta^2)(r_0-2\lfloor a_0 \rfloor) -2m\bigg) +(1+\beta)(r_0-2\lfloor a_0 \rfloor)-2m_1 \notag\\
                                &>(r_0-2\lfloor a_0 \rfloor)\bigg( (1-\beta^2)(r_0-2\lfloor a_0 \rfloor) -2m\bigg) +(1-\beta^2)(r_0-2\lfloor a_0 \rfloor)-2m_1 \notag \\
                                &= (r_0-2\lfloor a_0 \rfloor+1)\bigg( (1-\beta^2)(r_0-2\lfloor a_0 \rfloor) -2m_1\bigg) \notag\\
                                &\geq 0 \label{F>0}
\end{align}
by \eqref{condition-II} and \eqref{condition-III}, and
\begin{align}
    \frac{dF(x)}{dx}\bigg|_{x=r_0-2\lfloor a_0 \rfloor} &=2(1-\beta^2)(r_0-2\lfloor a_0 \rfloor)-(2m_1-\beta-1) \notag\\
    &\ge 4m_1-(2m_1-\beta-1)=2m_1+\beta+1>0 \label{F'>0}
\end{align}
by \eqref{condition-III}.
If $q'(u)< \beta q(u)$, then \eqref{eq:edges-u+Q(u)} implies
\begin{align} \label{eq:q-function}
    (1-\beta^2)q^2(u)-(2m_1-\beta-1)q(u)-2m_1<0.
\end{align}
Given that $q(u) \geq r_0-2\lfloor a_0 \rfloor$, \eqref{F>0} and \eqref{F'>0} imply $F(q(u)) \geq F(r_0-2\lfloor a_0 \rfloor) > 0$, contradicting \eqref{eq:q-function}.
\end{proof}

By the definition of $\mathcal{A}$, for each $V_i \in \mathcal{A}$, there exists a shortest $(V_i, V_1)$-path in $\mathcal{D}$. Among all such $V_i$, select one whose shortest $(V_i, V_1)$-path has maximal length. Evidently, $V_i$ is a terminal, as if it were not, there would exist a class $V_j \in \mathcal{A}$ with  $j\neq i$ for which the shortest $(V_j, V_1)$-path would necessarily traverse $V_i$, thereby increasing the length of that path. Without loss of generality, assume that $V_a$ is a terminal, where $1\leq a\leq \lfloor a_0 \rfloor$ by \eqref{eq:good-bound-a}.

Let $\mathcal{C}\subseteq \mathcal{B}$ be a strong component in $\mathcal{D}[\mathcal{B}]$ such that $|\mathcal{C}|\geq r-\lfloor a_0 \rfloor$. Such $\mathcal{C}$ exists by Claim \ref{claim:strong-component}. Let $C=\bigcup \mathcal{C}$ and let $c=|\mathcal{C}|$.
By \eqref{condition-II}, $|\mathcal{C}|\geq r-\lfloor a_0 \rfloor\geq r_0-\lfloor a_0 \rfloor>\lfloor a_0 \rfloor\geq a=|\mathcal{A}|$.
This information is powerful because if we can move some vertex $v\in V_j\in \mathcal{C}$ to $A$ and $A\cup \{v\}$ can be partitioned into $a$ independent classes each with exactly $s$ vertices, then we arrive at a new equitable $r$-coloring $\phi$ of $H$ so that $V_j-v$ is a small class, and for all classes $V_k\in \mathcal{C}-V_j$, there is a $(V_k,V_j-v)$-path in $\mathcal{D}(\phi)$ since $\mathcal{C}$ is a strong component. This implies $a(\phi)\geq |\mathcal{C}|>a=a(\varphi)$, contradicting the choice of $\varphi$.
We call such a vertex $v$ a \textit{friendly vertex}.

In order to find such a friendly vertex $v$, we want a class $V_i\in \mathcal{A}$ such that $e_H(v,V_i-u)=0$ for some $u\in V_i$ (this guarantees $V_i+v-u$ is still an independent set). Therefore, finding a solo edge $uv$ with $u\in V_i\in \mathcal{A}$  and
$v\in V_j\in \mathcal{C}$ would be helpful.

\begin{claim} \label{claim:Q(u)-notin-C}
    If $u\in V_i\in \mathcal{A}$ with $i\neq 1$ is movable to some other class  $V_{k}\in \mathcal{A}$, then $Q(u)\subseteq B\setminus C$.
\end{claim}

\begin{proof}
 Suppose for a contradiction that there is a solo edge $uv$ such that $v\in V_j\in \mathcal{C}$. 
 Move $v$ to $V_i$, move $u$ to $V_k$, and then switch witnesses along the $(V_k,V_1)$-path in $\mathcal{D}$ whenever $k\neq 1$. This gives an equitable $a$-coloring of $H[A\cup \{v\}]$. Therefore, $v$ is a friendly vertex, and as mentioned above, our proof is complete.
\end{proof}

For each edge $uv$ with $u\in V_i\in \mathcal{A}$ and $v\in B$, define its \textit{weight} as
\begin{align} \label{eq:weight-uv}
 \omega(uv) := 
\begin{cases}
    \frac{1}{e_H(v,V_i)} & \text{if } v \in B \setminus C, \\
    \frac{2}{e_H(v,V_i)} & \text{if } v \in C.
\end{cases}
\end{align}
The \textit{weight}  of a vertex $u\in V_i\in \mathcal{A}$ is defined by
\begin{align*}
    \omega(u)=\sum_{v\in N_B(u)}\omega(uv).
\end{align*}
It immediately follows
\begin{align}
   \notag \sum_{u\in V_i}\omega(u)=\sum_{u\in V_i}\sum_{v\in N_B(u)}\omega(uv)&=\sum_{\substack{ u\in V_i \\ v \in B \setminus C \\ uv \in E(H)  }} \omega(uv)+\sum_{\substack{ u\in V_i \\ v \in  C \\ uv \in E(H) }} \omega(uv)\\
 \label{eq:weight-sum} &= |B\setminus C|+2|C|=(b-c)s+2cs=(b+c)s.
\end{align}
Since $b=r-a\geq r-\lfloor a_0 \rfloor$, $c\geq r-\lfloor a_0 \rfloor$ by Claim \ref{claim:strong-component}, and $s-1\leq |V_i|\leq s$,
we conclude from \eqref{eq:weight-sum} that the average weight of vertices in $V_i$ is at least $2r-2\lfloor a_0 \rfloor$. Hence taking $i=a$, it is possible to select a special vertex $u^*$ in the terminal $V_a$ so that 
\begin{align} \label{eq:u*lowerbound}
    \omega(u^*)\geq 2r-2\lfloor a_0 \rfloor.
\end{align}

We now estimate the upper bound of $\omega(u^*)$. First, 
if $Q(u^*)\subseteq B\setminus C$, then 
\begin{align*}
    \omega(u^*)\leq q(u^*)+\max\left\{\frac{1}{2},\frac{2}{2}\right\} \cdot (\deg_G(u^*)-q(u^*))\leq r<2r-2\lfloor a_0 \rfloor
\end{align*}
by \eqref{condition-II}, contradicting \eqref{eq:u*lowerbound}.
Hence $Q(u^*)\cap C\neq \emptyset$. Moreover, by Claim \ref{claim:Q(u)-notin-C}, either $a=1$ or $a\geq 2$ and
$u^*$ is not movable to any other class of $\mathcal{A}$, in which case there exists another vertex $u^{**}\in V_a$ that is movable to some other class of $\mathcal{A}$ since
$V_a$ is accessible.

Let $\ell_1:=|Q(u^*)\cap C|$ and $\ell_2:=|Q(u^*)\setminus C|$.
For each neighbor $v$ of $u^*$ in $B$, $\omega(u^*v)$ is either $2$ if $v \in Q(u^*)\cap C$, $1$ if $v \in Q(u^*)\setminus C$, $\leq 1$ if $v \notin Q(u^*)\cap C$, or $\leq 2$ if $v \notin Q(u^*)\setminus C$, as specified in \eqref{eq:weight-uv}.
Therefore,
\begin{align*}
    \omega(u^*)&\leq 2\ell_1+(\deg_G(u^*)-\ell_1)\leq r+\ell_1,\\
    \omega(u^*)&\leq \ell_2+2(\deg_G(u^*)-\ell_2)\leq 2r-\ell_2.
\end{align*}
Combining them with \eqref{eq:u*lowerbound}, we have 
\begin{align} 
    \label{eq:qu} q(u^*) &\geq \ell_1\ge r_0-2\lfloor a_0 \rfloor,\\
    \label{eq:qu-not-in-C} |Q(u^*)\setminus C|&= \ell_2\le 2\lfloor a_0 \rfloor. 
\end{align}


By \eqref{eq:qu-not-in-C}, there are at most $2\lfloor a_0 \rfloor$ nice solo neighbors of $u^*$ in $B\setminus C$.
By \eqref{eq:qu}, \eqref{condition-IV} and by Claim  \ref{claim:counting-nice-solo}, we have $q'(u^*)\geq \beta(r_0-2\lfloor a_0 \rfloor)>2a_0$.
Thus, there is a class $V_j\in \mathcal{C}$
containing at least one nice solo neighbor $v$ of $u^*$.
By the definition of nice solo neighbor, there is another class $V_i\in \mathcal{B}$ (anyway it is possible $i=j$) containing another nice solo neighbor $w$ of $u^*$ such that $vw\not\in E(H)$.

Let $N_0,N_1$ and $N_{2+}$ be the number of classes in $\mathcal{B}$ containing zero, exactly one, and at least two neighbors of $u^*$.
Since $N_0+N_1+N_{2+}=|\mathcal{B}|=r-a$, $N_1+2N_{2+}\leq \deg_G(u^*)\leq r$, and $N_{2+}\geq \big\lceil \frac{q'(u^*)}{2}\big\rceil$ by Claim \ref{claim:solo-at-least-2-neighbors-}, we have
\begin{align*}
  \notag  N_0 &=(r-a)-N_1-N_{2+}\geq (r-a)-(r-2N_{2+})-N_{2+}\geq N_{2+}-a\geq \frac{q'(u^*)}{2}-a 
              >  a_0-a \ge \lfloor a_0 \rfloor-a
\end{align*}
by \eqref{condition-IV}.
Since $|\mathcal{B}\setminus \mathcal{C}|=b-c\leq (r-a)-(r-\lfloor a_0 \rfloor)=\lfloor a_0 \rfloor-a$ by Claim \ref{claim:strong-component}, we have
$N_0-|\mathcal{B}\setminus \mathcal{C}|>0$.
This means that there is a class $V_k\in \mathcal{C}$ that does not contain any neighbor of $u^*$. In other words, $u^*$ is movable to $V_k$.

Since $V_j,V_k\in \mathcal{C}$ and $C$ is a strong component, there is a $(V_k,V_j)$-path $P$ in $\mathcal{D}[\mathcal{C}]$.
We assume that $P$ does not pass $V_i$ because otherwise $V_i\in \mathcal{C}$ and we exchange the roles of $V_i$ and $V_j$.
At this stage, we move $u^*$ to $V_k$, switch witnesses along $P$, and move $v$ to $V_a$.
This gives a new equitable $r$-coloring $\phi$ of $H$ with small class $V_1$ if $a\geq 2$ or with small class $V_a+v-u^*$ if $a=1$.
Since $V_a$ is a terminal under $\varphi$, each $V_{h}\in \mathcal{A}$ with $h\neq a$ is still an accessible class under $\phi$.
Due to the existence of $u^{**}$ whenever $a\geq 2$ (recall that $u^{**}$ is a vertex in $V_a$ that is movable to some other class of $\mathcal{A}$), $\overrightarrow{(V_a+v-u^*)V_{h}}$ is an arc in $\mathcal{D}(\phi)$ for some $h\neq a$, so 
$V_a+v-u^*$ is accessible under $\phi$. Since $vw\not\in E(H)$ and $u^*w$ is a solo edge, $e_H(w,V_a+v-u^*)=0$. This implies that
$\overrightarrow{V_i(V_a+v-u^*)}$ is an arc in $\mathcal{D}(\phi)$, and therefore, $V_i$ is accessible under $\phi$ (the conclusion holds even for $a=1$).
Now, $a(\phi)\geq a+1= a(\varphi)+1$, contradicting the choice of $\varphi$.

\section{Concluding Remarks}

The \textit{local crossing number for a particular drawing} of a graph is the highest number of crossings on any single edge, whereas the \textit{local crossing number of the graph itself} is the lowest of these values across all possible drawings. We use ${\rm LCR}(G)$ denote the local crossing number of the graph $G$. 

A graph $G$ is \textit{$k$-planar} if  ${\rm LCR}(G)\leq k$. Among graphs beyond planarity, $k$-planar graphs, particularly those with $k=1$, had been extensively studied in the literature. For further reading, refer to the following citations: \cite{zbMATH07796524,zbMATH07437222,zbMATH07549465,zbMATH07374188,zbMATH07374185,zbMATH07374182,zbMATH07662546,zbMATH07746549,zbMATH07745264,zbMATH07500683,zbMATH07613730,zbMATH07383905,zbMATH06723133}.
From a combinatorial perspective, one primary question focuses on determining the maximum number of edges a graph can have within a specific graph class. 
Such question in extreme graph theory are typically referred to as Tur\'an-type problems \cite{10.5555/7228}. For several graph classes, particularly those with a bounded local crossing number, precise density (general or bipartite) upper bounds had been established. The following table\footnote{
From this table, we observe that all the graph classes mentioned fall within the category of $\mathcal{G}_{m_1,m_2}$ where the ratio $m_1/m_2$ is less than or equal to 1.8. Consequently, incorporating the condition $m_1/m_2 \leq 1.8$ in Theorem~\ref{thm-main} is justified in this context.}
collects such type of some known results, where $n$ denotes the number of vertices in a given graph. 

\begin{center}
\begin{tabular}{|c | c c | c c|}
\hline
graph classes & \multicolumn{2}{c|}{general density} & \multicolumn{2}{c|}{bipartite density} \\ \cline{2-3} \cline{4-5}
 & upper bound & Ref. & upper bound & Ref. \\ \hline
planar graphs & $3n-6$ & \cite{diestel2017graph} & $2n-4$ & \cite{diestel2017graph} \\ \hline
1-planar graphs & $4n-8$ & \cite{zbMATH03214386} & $3n-8$ & \cite{zbMATH06535852} \\ \hline
2-planar graphs & $5n-10$ & \cite{zbMATH01226449} & $3.5n-7$ & \cite{zbMATH07561382} \\ \hline
3-planar graphs & $5.5n-11$ & \cite{zbMATH06687308} & $5.205n$ & \cite{zbMATH07561382} \\ \hline
4-planar graphs & $6n-12$  & \cite{zbMATH07194817} & --- & --- \\ \hline
$k$-planar graphs ($k\geq 5$) & $3.81\sqrt{k}n$ & \cite{zbMATH07194817} & $3.005\sqrt{k}n$ & \cite{zbMATH07561382} \\ \hline
\end{tabular}
\end{center}

Let $\mathcal{Y}_k$ denote the class of $k$-planar graphs, and let $\Delta_k$ be the smallest integer such that any graph $G \in \mathcal{Y}_k$ with $\Delta(G) \geq \Delta_k$ admits an equitable $r$-coloring for all $r \geq \Delta(G)$.
As stated at the beginning of this paper, Kostochka, Lin, and Xiang \cite{zbMATH07840571} had shown that $\Delta_0 \leq 8$, and Cranston and Mahmoud \cite{CRANSTON2025114286} had proved that $\Delta_1 \leq 13$. Based on the density results presented in the table above, we observe that:
 
- $\mathcal{Y}_2 \in \mathcal{G}_{5,3.5}$,

- $\mathcal{Y}_3 \in \mathcal{G}_{5.5,5.205}$,

- $\mathcal{Y}_4 \in \mathcal{G}_{6,6}$.

\noindent Thus, applying Theorem \ref{thm-main}, we obtain the following results.

\begin{cor} \label{cor-1}
    Let $G$ be a graph with maximum degree $\Delta$. If 
    
    - $G$ is $2$-planar and $\Delta\geq 24$, or
    
    - $G$ is $3$-planar and $\Delta\geq 33$, or
       
    - $G$ is $4$-planar and $\Delta\geq 38$, 
    
    \noindent then $G$ has an equitable $r$-coloring for every $r\geq \Delta$.
Alternatively, this can be expressed as: 

\begin{center}
    $\Delta_2\leq 24$, $\Delta_3\leq 33$, and $\Delta_4\leq 38$.
\end{center}
\end{cor}

 Given that every $n$-vertex $k$-planar graph with $k \geq 2$ has at most $3.81\sqrt{k}n$ edges \cite{zbMATH07194817}, and every $n$-vertex bipartite $k$-planar graph with $k \geq 1$ has at most $3.005\sqrt{k}n$ edges \cite{zbMATH07561382}, we can deduce the following general result using Theorem \ref{thm-main}
 
\begin{cor} \label{cor-2}
    If $G$ is a $k$-planar graph with $k\geq 2$ and $\Delta(G)\geq 19.57 \sqrt{k}$,
    then $G$ has an equitable $r$-coloring for every $r\geq \Delta(G)$.
In other words, we have $$\Delta_k\leq 19.57\sqrt{k}$$ for every $k\geq 2$.
\end{cor}

By integrating Corollary \ref{cor-2} with the pivotal discoveries of Cranston and Mahmoud \cite{CRANSTON2025114286} concerning 1-planar graphs, and those of Kostochka, Lin, and Xiang \cite{zbMATH07840571} related to planar graphs, we arrive at the subsequent conclusion.

\begin{cor} \label{cor-3}
    If $G$ is a graph with ${\rm LCR}(G)\leq \Delta(G)^2 / 383$,
    then $G$ has an equitable $r$-coloring for every $r\geq \Delta(G)$.
\end{cor}

For further research, we consider two notable subclasses of 1-planar graphs: \textit{IC-planar graphs}, where crossings are independent (meaning no two crossed edges have a common endpoint), and \textit{NIC-planar graphs}, where crossings are nearly independent (indicating that no two pairs of crossed edges share two endpoints).
IC-planar graphs were initially investigated by Kr{\'a}l' and Stacho \cite{zbMATH05814476} under the term "plane graphs with independent crossings", but many recent researches (e.g.,\,\cite{zbMATH07383905,zbMATH06888432}) have adopted the terminology of IC-planar graphs, as encouraged by Zhang \cite{zbMATH06378992}, who also introduced the notion of NIC-planar graphs in the same work. The following table gives the density results for $n$-vertex IC-planar graphs and NIC-planar graphs.

\begin{center}
\begin{tabular}{|c | c c | c c|}
\hline
graph classes & \multicolumn{2}{c|}{general density} & \multicolumn{2}{c|}{bipartite density} \\ \cline{2-3} \cline{4-5}
 & upper bound & Ref. & upper bound & Ref. \\ \hline
IC-planar graphs & $3.5n-7$ & \cite{zbMATH06137073} & $2.25n-4$ & \cite{zbMATH07561382} \\ \hline
NIC-planar graphs & $3.6n-7.2$ & \cite{zbMATH06378992} & $2.5n-5$ & \cite{zbMATH07561382} \\ \hline
\end{tabular}
\end{center}

Cranston and Mahmoud \cite{CRANSTON2025114286} established that 1-planar graphs $G$ with a maximum degree $\Delta \geq 13$ can be equitably $r$-colored for any $r \geq \Delta$. However, when Theorem \ref{thm-main} is applied in conjunction with the density results given by the above table, a lower bound for the maximum degree less than 13 cannot be established for either IC-planar or NIC-planar graphs. 
Specifically, our analysis yields a bound of 16 for IC-planar graphs and 17 for NIC-planar graphs. The underlying reason for this inability to achieve a smaller bound is rooted in the proof of Theorem \ref{thm-mainfinal}. Notably, during the proof, we did not account for varying sizes of $|\mathcal{A}|$ by considering separate cases.
Addressing this issue in the proof could potentially lead to a maximum degree bound that is better than 13 for both IC-planar and NIC-planar graphs. To this end, the first author demonstrated that any IC-planar graph or NIC-planar graph $G$ with $\Delta(G) \geq 10$ or $\Delta(G) \geq 11$, respectively, can be equitably $r$-colored for all $r \geq \Delta(G)$. This represents an advancement over the findings of Zhang \cite{zbMATH06602931} and Zhang, Wang, and Xu \cite{zbMATH06875705}, as it reduces the threshold from 12 to 10 for IC-planar graphs and from 13 to 11 for NIC-planar graphs. The proofs of these two results are presented in a separate article \cite{liu2025}.

\bibliographystyle{alpha}
\bibliography{ref}

\end{document}